\documentclass{amsart}
\usepackage{hyperref}
\usepackage{graphicx}
\usepackage{enumitem}                               
\usepackage{amssymb}
\usepackage{amsthm}
\newtheorem{theorem}{Theorem}[section] 

\newtheorem{lemma}[theorem]{Lemma}
\newtheorem{conjecture*}[theorem]{Conjecture}
\newtheorem{proposition}[theorem]{Proposition}

\newtheorem{corollary}[theorem]{Corollary}
\newtheorem{question}[theorem]{Question}
\newtheorem*{conclusion*}{Conclusion}

\newtheoremstyle{notauto}{}{}{\itshape}{}{\bfseries}{.}{0.5em}{\thmnote{#3}}
\theoremstyle{notauto}

\theoremstyle{definition}
\newtheorem{definition}[theorem]{Definition}

\theoremstyle{remark}
\newtheorem{remark}[theorem]{Remark}

\begin{document}
	
\baselineskip13.5pt
	
\title[  A Generalization of Hall-Wielandt Theorem]{ A Generalization of Hall-Wielandt Theorem }
\author{{M.Yas\.{I}r} K{\i}zmaz }
\address{Department of Mathematics, Middle East Technical University, Ankara 06531, Turkey}
\email{yasir@metu.edu.tr}
\begin{abstract}
	Let $G$ be a finite group and $P\in Syl_p(G)$. We denote the $k$'th term of the upper central series of $G$ by $Z_k(G)$ and the norm of $G$ by $Z^*(G)$. In this article, we prove that if for every tame intersection $P\cap Q$ such that $Z_{p-1}(P)<P\cap Q<P$, the group $N_G(P\cap Q)$ is $p$-nilpotent  then $N_G(P)$ controls $p$-transfer in $G$.
  For $p=2$, we sharpen our results by proving if for every tame intersection $P\cap Q$ such that $Z^*(P)<P\cap Q<P$, the group $N_G(P\cap Q)$ is $p$-nilpotent then $N_G(P)$ controls $p$-transfer in $G$. We also obtain several corollaries which give sufficient conditions for $N_G(P)$ to controls $p$-transfer in $G$ as a generalization of some well known theorems, including Hall-Wielandt theorem and Frobenius normal complement theorem.
\end{abstract}
\maketitle	

\section{introduction}
Throughout the article, we assume that all groups are finite. Notation and terminology are standard as in \cite{1}. Let $G$ be a group and $P\in Syl_p(G)$. We say that $G$ is \textit{$p$-nilpotent} if it has a normal Hall $p'$-subgroup. Let $N$ be a subgroup of $G$ such that $|G:N|$ is coprime to $p$. Then $N$ is said to \textit{control $p$-transfer} in $G$ if $N/A^p(N)\cong G/A^p(G)$. A famous result of Tate in \cite{2} shows that $N/A^p(N)\cong G/A^p(G)$ if and only if $N/O^p(N)\cong G/O^p(G)$. Thus, $N$ controls $p$-transfer in $G$ if and only if $N/O^p(N)\cong G/O^p(G)$. In this case, one can also deduce that $N$ is $p$-nilpotent if and only if $G$ is $p$-nilpotent.

 By a result due to Burnside,
  $N_G(P)$ controls $p$-transfer in $G$ if $P$ is abelian. Later works of Hall and Wielandt showed that $N_G(P)$ controls $p$-transfer in $G$ if the class of $P$ is not ``too large". Namely, they proved the following generalization of Burnside's result.
\begin{theorem}[Hall-Wielandt]
	If the class of $P$ is less than $p$, then $N_G(P)$ controls $p$-transfer in $G$.
\end{theorem}

 In 1958, Yoshida introduced the concept of character theoretic transfer and by the means of it, he obtained the following generalization of Hall-Wielandt theorem.

\begin{theorem}\cite[Theorem 4.2]{7}\label{isc}
	If $P$ has no quotient isomorphic to $\mathbb Z_p \wr \mathbb Z_p$ then $N_G(P)$ controls $p$-transfer in $G$.
\end{theorem}

 The original proof of this strong theorem depends on character theory. However, Isaacs provided a character free proof to Yoshida's theorem in his book (see section 10 in \cite{1}). Taking advantages of his method, we obtain another generalization of Hall-Wielandt theorem.
 
 Before presenting our main theorem, it is convenient here to give some conventions that we adopt throughout the paper.
Let $P,Q\in Syl_p(G)$ (possibly $P=Q$). We say that $P\cap Q$ is a \textbf{tame intersection} if both $N_P(P\cap Q)$ and $N_Q(P\cap Q)$ are Sylow $p$-subgroups of $N_G(P\cap Q)$. For simplicity, we use directly ``$X\cap Y$ is a tame intersection" without specifying what $X$ and $Y$ are. In this case, it should be understood that $X$ and $Y$ are Sylow $p$-subgroups of $G$ for a prime $p$ dividing order of $G$ and $X\cap Y$ is a tame intersection according to the formal definition.
  
 The following is the main theorem of our article. 

\begin{theorem}\label{main theorem}
 Assume that for each tame intersection $Z_{p-1}(P) <P\cap Q<P$, the group $N_G(P\cap Q)$ is $p$-nilpotent. Then $N_G(P)$ controls $p$-transfer in $G$.
\end{theorem}

The next remark shows that our theorem extends the result of Hall-Wielandt theorem in a different direction than what Yoshida's theorem does.

\begin{remark}
	Let $G$ be a group having a Sylow $p$-subgroup $P$ isomorphic to $\mathbb Z_p \wr \mathbb Z_p$. Clearly, Yoshida's theorem is not applicable here. If $N_G(P)$ does not control $p$-transfer in $G$ then there exists a Sylow $p$-subgroup $Q$ of $G$ such that $|P:P\cap Q|=p$ and $N_G(P\cap Q)$ is not $p$-nilpotent by Theorem \ref{main theorem}. Notice that this is exactly the case where $G=S_4$ and $p=2$.  We can say in other way that $N_G(P)$ controls $p$-transfer in $G$ if $|P:P\cap P^x|> p$ for each $x\in G\setminus N_G(P)$.
\end{remark}
Some of the immediate corollaries of Theorem \ref{main theorem} are as follows.
\begin{corollary}\label{main cor}
	Assume that for any two distinct Sylow $p$-subgroups $P$ and $Q$ of $G$, the inequality $|P\cap Q|\leq |Z_{p-1}(P)|$ is satisfied. Then $N_G(P)$ controls $p$-transfer in $G$.
\end{corollary}

 The next corollary is a generalization of the well-known Frobenius normal complement theorem, which guarantees the $p$-nilpotency of $G$ if $N_G(X)$ is $p$-nilpotent for each nontrivial $p$-subgroup $X$ of $P$.
\begin{corollary}\label{gen Frob}
	 Assume that for each tame intersection $Z_{p-1}(P) <P\cap Q$, the group $N_G(P\cap Q)$ is $p$-nilpotent. Then $G$ is $p$-nilpotent.
\end{corollary}

\begin{remark}
	The main ingredient in proving most of the $p$-nilpotency theorems including Thompson-Glauberman $p$-nilpotency theorems is the Frobenius normal complement theorem, and hence its above generalization can be used in proving more strong $p$-nilpotency theorems.
\end{remark}

When $p=2$, Theorem \ref{main theorem} guarantees that  if $N_G(P\cap Q)$ is $p$-nilpotent for each tame intersection $Z(P) <P\cap Q<P$,  then $N_G(P)$ controls $p$-transfer in $G$. In fact, we shall extend this result further.

Let $Z^*(P)$ denote the norm of $P$, which is defined as $$Z^*(P):=\bigcap_{H\leq P}  N_P(H).$$ We have clearly $Z(P)\leq Z^*(P)$. One can recursively define $Z^*_i(P)$ for $i\geq 1$ as the full inverse image of $Z^*(P/Z^*_{i-1}(P))$ in $P$ and set $Z^*_0(P)=1$. We also say that $P$ is of norm length at most $i$ if $Z^*_i(P)=P$. We should also note that it is well known that $Z^*(P)$ is contained in the second center of $P$.

\begin{theorem}\label{Z version}
	Assume that for each tame intersection $Z^*(P) <P\cap Q<P$, the group $N_G(P\cap Q)$ is $p$-nilpotent. Then $N_G(P)$ controls $p$-transfer in $G$.
\end{theorem}

The following corollary is stronger than Corollary \ref{gen Frob} when $p=2$ although it is also true for odd primes (as Theorem \ref{Z version} is also true for odd primes).

\begin{corollary}\label{gen frob Z-version}
	Assume that for each tame intersection $Z^*(P) <P\cap Q$, the group $N_G(P\cap Q)$ is $p$-nilpotent. Then $G$ is $p$-nilpotent.
\end{corollary}

The following theorem is a generalization of a theorem due to Gr\"{u}n (see Theorem 14.4.4 in \cite{4}),  which states that the normalizer of a $p$-normal subgroup controls $p$-transfer in $G$. We also use our next theorem in the proof of Theorem \ref{Z version}.

\begin{theorem}\label{nwthm}
	Let $K\leq Z^*(P)$ be a weakly closed subgroup of $P$. Then $N_G(K)$ controls $p$-transfer in $G$.
\end{theorem}

The next corollary can also be  easily deduced by the means of Theorem \ref{Z version}.
\begin{corollary}
	Assume that for any two distinct Sylow $p$-subgroups $P$ and $Q$ of $G$, the inequality $|P\cap Q|\leq |Z^*(P)|$ is satisfied. Then $N_G(P)$ controls $p$-transfer in $G$.
\end{corollary}
\begin{remark}
In above theorems, the assumption "$N_G(P\cap Q)$ is $p$-nilpotent" could be replaced with a weaker assumption "$N_G(P\cap Q)/C_G(P\cap Q)$ is a $p$-group". This can be observed with the proofs of Theorems \ref{main theorem} and \ref{Z version}.
\end{remark}

\section{Preliminaries}


	Let $H\leq G$ and $T=\{t_i \mid i=1,2\ldots, n \}$ be a right transversal for $H$ in $G$. The map $V:G\to H$ defined by $$V(g)=\prod_{i=1}^n t_ig(t_i.g)^{-1}  $$ is called  \textbf{a pretransfer map} from $G$ to $H$. When the order of the product is not needed to specify, we simply write $V(g)=\prod_{t\in T} tg(t.g)^{-1}$. Notice that the kernel of ``dot action" is $Core_G(H)$, and so $t.g=t$ for all $g\in Core_G(H)$. In the case that $G$ is a $p$-group, $Z(G/ Core_G(H))\neq 1$ whenever $H$ is a proper subgroup of $G$. If $x\in G$ such that $xCore_G(H)\in Z(G/ Core_G(H))$ of order $p$, then each $\langle x \rangle$-orbit has length $p$ when we consider the action of $\langle x \rangle$ on $T$. 
	
	Let $t_1,t_2\ldots,t_k$ be representatives of all distinct orbits of $\langle x \rangle$ on $T$. As $t.x$ and $tx$ represent the same right coset of $H$ in $G$ for each $t\in T$, the set $T^*=\{t_ix^j \mid i\in\{1,2,...,k\} \ and \ j\in \{0,1,...,p-1\}\}$ is also a right transversal for $H$ in $G$. Let $V^*$ be a pretransfer map constructed by using $T^*$. Since $V(u)\equiv V^*(u) \ mod \ H'$, we may replace $T$ with $T^*$ without loss of generality whenever such a situation occurs.
	
We denote all pretransfer maps with upper case letters and each corresponding lower case letter shows the corresponding transfer map.

\begin{theorem}\cite[Theorem 10.8]{1}\label{transivity}
Let $G$ be a group, and suppose that $H\leq K\leq G$. Let
$U:G\to K$, $W : K \to H$ and $V :G\to H$
be pretransfer maps. Then for all $g\in G$, we have $V(g) \equiv W(U(g)) \ mod \ H'$, that is, $v(g)=w(U(g))$.
\end{theorem}

\begin{theorem}\label{mackeytransfer}\cite[Theorem 10.10]{1}  Let $X$ be a set of representatives for the $(H,K)$ double cosets in a group $G$, where $H$ and $K$ are subgroups of $G$. Let $V: G \to H$ be a pretransfer map, and for each element $x\in  X$, let $W_x :K\to K\cap H^x$ be a pretransfer map. Then for $k\in K$, we have $$V(k)\equiv \prod\limits_{x\in X}xW_x(k)x^{-1} \ mod \ H'.$$
\end{theorem}

Now we give a technical lemma, which is essentially the method used in the proof of Yoshida's theorem (see proof of Theorem 10.1 in \cite{1}). For the sake completeness, we give the proof of this lemma here.

\begin{lemma} \label{basic}
Let $G$ be a group and, let $P\in Syl_p(G)$ and $N_G(P)\leq N$. Suppose that $N$ does not control $p$-transfer in $G$ and let $X$ be a set of representatives for the $(N,P)$ double cosets in $G$, which contains the identity $e$. Then the following hold:

\begin{enumerate}[label=(\alph*)]
	\item There exists a normal subgroup $M$ of $N$ of index $p$ such that $V(G)\subseteq M$ for every pretransfer map $V$ from $G$ to $N$.
	\item For each $u\in P\setminus M$, there exists a nonidentity $x\in X$ such that $W(u)\notin P\cap M^x$ where $W$ is a pretransfer map from $P$ to $P\cap N^x$.
	\item For the $x$ in part (b), we have $P\cap N^x<P$ and  $|P\cap N^x:P\cap M^x|=p$.
\end{enumerate}
	\end{lemma}

\begin{proof}[\textbf{Proof}]
	\begin{enumerate}[label=(\alph*)]
		\item It follows by (\cite{1}, Lemma 10.11).
		\item Let $u\in P\setminus M$. Let $W_x$ be a pretransfer map from $P$ to $P\cap N^x$ for each $x\in X$. Then we have
		$$V(u)\equiv \prod_{x\in X}xW_x(u)x^{-1} \ mod \ N'$$
		by Theorem \ref{mackeytransfer}.
		Since $N'\leq M$ and $V(u)\in M$, we get $$\prod_{x\in X}xW_x(u)x^{-1}\in M.$$ Notice that for $x=e$, $W_e:P\to P$ and $W_e(u)=u=eW_e(u)e^{-1}\notin M$. Thus, there also exists $e\neq x\in X$ such that $xW_x(u)x^{-1}\notin M$. Set $W_x=W$. Then we get $W(u)\in P\cap N^x\setminus P\cap M^x.$ 
		
		\item Set $R=P\cap N^x$ and $Q= P\cap M^x$. If $R=P$ then $P^{x^{-1}}\leq N$, and hence there exists $y\in N$ such that $P^{x^{-1}y}=P$. Since $x^{-1}y\in N_G(P)\leq N$, we get $x\in N$. This is not possible as $NxP=NeP$ and $x\neq e$. It follows that that $R<P$. Note that $R\neq Q$ by part $(b)$. Moreover, the inequality $1<|R:Q|\leq |N^x:M^x|=p$ forces that $|R:Q|=p$.
	\end{enumerate}
\end{proof}

\section{Main Results}

The following lemma serves as the key tool in proving our main theorems since it enables us to use induction in the proof ``control $p$-transfer theorems". Throughout the section, $G$ is a group and $P$ is a Sylow $p$-subgroup of $G$ for a prime $p$ dividing the order of $G$.

\begin{lemma}\label{main lemma}
	 Let $N_G(P)\leq N\leq G$, $Z\leq P$ and $Z\lhd G$. Assume that $N/Z$ controls $p$-transfer in $G/Z$ and that one of the following holds:
	\begin{enumerate}[label=(\alph*)]
		\item  $[Z,g,\ldots,g]_{p-1}\leq \Phi(Z)$ for all $g\in P$. 
		\item $Z\leq \Phi(P)$. 
	\end{enumerate}
	Then $N$ controls $p$-transfer in $G$. 
\end{lemma}

We need the following lemma in the proof of Lemma \ref{main lemma}.
\begin{lemma}\label{u,Z}
 Let $N_G(P)\leq N\leq G$, $Z\leq P$ and $Z\lhd G$. Assume that $N$ does not control $p$-transfer in $G$ and $N/Z$ controls $p$-transfer in $G/Z$. Then $Z \nsubseteq M$ and  we have $W(u) \in P \cap N^x \setminus P\cap M^x$ for each $u\in Z\setminus M $ where $W,M$ and $x$ are as in Lemma \ref{basic}.
\end{lemma}

\begin{proof}[\textbf{Proof}]
	Set $G/Z=\overline G$. Let $V$ be a pretransfer map from $G$ to $N$. Let $T$ be a right transversal set used for constructing $V$. It follows that there exist a normal subgroup $M$ of $N$ with index $p$ such that $V(G)\subseteq M$ by Lemma \ref{basic}(a).
	
	Now we claim that $Z\nsubseteq M$. Assume to the contrary. Notice that the set $\overline T=\{\overline t\mid t\in T \}$ is a right transversal set for $\overline N$ in $\overline G$. Thus if we construct a pretransfer map $\overline V$ by using $\overline T$, then $\overline V(\overline g)=\overline {V(g)}$. It follows that $\overline V(\overline G)=\overline{V(G)}\subseteq \overline M \lhd \overline N$. Let $W$ be a pretransfer map from $\overline N$ to $\overline P$. Note that $ker(w)=A^p(\overline N) \leq \overline  M$ as $|\overline N : \overline M|=p$, and hence $w(\overline M)<w(\overline N)$. It then follows that $w(\overline V(\overline G))<w(\overline N)$. Since $w\circ \overline V$ is the transfer map from $\overline G$ to $\overline P$ by Theorem \ref{transivity}, we get $|\overline G:A^p(\overline G)|\neq |\overline N:A^p(\overline N)|$, which contradicts the hypothesis. Thus there exists $u\in Z$ such that $u\in N\setminus M$. Then we have $W(u) \in P \cap N^x \setminus P\cap M^x$ for each $u\in Z\setminus M $
	by Lemma \ref{basic}(b).
\end{proof}

\begin{proof}[\textbf{Proof of Lemma \ref{main lemma}}]
	Assume that $N$ does not control $p$-transfer in $G$. We derive contradiction for both parts.
		
	First assume that $(b)$ holds, that is, $Z\leq \Phi(P)$. Note that $|P:P\cap M|=p$, and so $Z\leq \Phi(P)\leq M\cap P$. However, this is not possible by Lemma \ref{u,Z}. This contradiction shows that $N$ controls $p$-transfer in $G$ when $(b)$ holds.
	
	Now assume that $(a)$ holds. Let $X$ be a set of representatives for the $(N,P)$ double cosets in $G$, which contains the identity $e$.
	By Lemma \ref{basic}(b), we have a pretransfer $W:P\to P\cap N^x$ such that $W(u)\notin P\cap M^x$ for some nonidentity $x\in X$. Set $R=P\cap N^x$ and $Q=P\cap M^x$.
	
	Now let $S$ be a right transversal set for $R$ in $P$ used for constructing $W$ so that we have $W(u)=\prod_{s\in S}su(s.u)^{-1}.$ Since $u\in Z\leq Core_P(R)$, we have $(s.u)=s$ for all $s\in S$. Thus we get $W(u)=\prod_{s\in S}sus^{-1}$.

	Set $C= Core_P(R)$. Since $R<P$ by Lemma \ref{basic}(c), $C$ is also proper in $P$. So we see that $Z(P/ C)\neq 1$. Now choose $n\in P$ such that $n C\in  Z(P/ C)$ of order $p$. Then each $\langle n \rangle$-orbit has length $p$. Let $s_1,s_2\ldots,s_k$ be representatives of all distinct orbits of $\langle n \rangle$ on $S$. Without loss of generality, we can suppose that $S=\{s_in^j \mid i\in\{1,2,...,k\} \ and \ j\in \{0,1,...,p-1\}\}$. Now we compute the contribution of a single $\langle n \rangle$-orbit to $W(u)$. Fix $s\in S$.
	$$(snun^{-1}s^{-1})(sn^2un^{-2}s^{-1})\ldots (sn^{p-1}un^{-p+1}s^{-1})(sus^{-1})=s(nu)^{p-1}n^{-p+1}us^{-1}.$$ 
	We have $s(nu)^{p-1}n^{-p+1}us^{-1}=(s(nu)^{p}s^{-1})(su^{-1}n^{-p}us^{-1}).$ Set $H=\langle n,u\rangle$. Due to the fact that $|\langle n\rangle C:C|=p$, we have $H'\leq C$. Note that $u\in Z\leq C$, and so $$[H',u]\equiv 1 \ mod \ \Phi(C).$$
	We can expand the power of the product as in the following form  $$(nu)^p\equiv (n^pu^p)[u,n]^{p\choose 2}[u,n,n]^{p\choose 3}...[u,n,...,n]_{p-2}^p[u,n,...,n]_{p-1}  \ mod \ \Phi(C)$$
	due to the previous congruence.

	As $C\lhd P$, we observe that $s[u,n,...,n]_is^{-1}\in C$ for $i=1,...,p-1$, and so $(s[u,n,...,n]_is^{-1})^p\in \Phi(C)$ for $i=1,...,p-1$.  By using the fact that $p\choose i+1$ is divisible by $p$ for $i=1,\ldots, p-2$, we see that $$(s[u,n,...,n]_{i}s^{-1})^{p\choose i+1}\in \Phi(C) \ for \ i=1, \ldots, p-2. $$
	Note also that $[u,n,...,n]_{p-1}\in \Phi(Z)\leq \Phi(C)$ by hypothesis, and so we get that $s[u,n,...,n]_{p-1}s^{-1}\in \Phi(C)$ since $\Phi(C)\lhd P$.
	As a consequence, we obtain that $$s(nu)^ps^{-1}\equiv (sn^ps^{-1})(su^ps^{-1}) \equiv sn^ps^{-1} \ mod \ \Phi(C).$$
	
	It then follows that $$(s(nu)^{p}s^{-1})(su^{-1}n^{-p}us^{-1})\equiv (sn^ps^{-1})(su^{-1}n^{-p}us^{-1}) \equiv s[n^{-p},u]s^{-1} \equiv 1 \ mod \ \Phi(C) .$$
	
	We only need to explain why the last congruence holds: Since both $n^{-p}$ and $u$ are elements of $C$, we see that $[n^{-p},u]\in \Phi(C)$. It follows that $s[n^{-p},u]s^{-1}\in \Phi(C)$ due to the normality of $\Phi(C)$ in $P$. Then $W(u)\in \Phi(C)$ as the chosen $\langle n \rangle$-orbit is arbitrary. Since $|R:Q|=p$ by Lemma \ref{basic}(c), the containment
	 $\Phi(C)\leq \Phi(R)\leq Q$ holds. As a consequence, $W(u)\in Q$. This contradiction completes the proof.
\end{proof}

\begin{remark}
	In the proofs of many $p$-nilpotency theorems, the minimal counter example $G$ is a $p$-soluble group such that $O_{p'}(G)=1$ and $G/O_p(G)$ is $p$-nilpotent. Lemma \ref{main lemma}(a) guarantees the $p$-nilpotency of $G$ if $[O_p(G),g,\ldots,g]_{p-1}\leq \Phi(O_p(G))$ for all $g\in P$. In particular if $O_p(G)\leq Z_{p-1}(P)$ then the $p$-nilpotency of $G$ follows. This bound seems to be best possible since in the symmetric group $S_4$, $O_2(S_4)\leq Z_2(P)$ and $O_2(G)\nleq Z(P)$. Even if  $S_4/O_2(S_4)$ is $2$-nilpotent, $S_4$ fails to be $2$-nilpotent.
	
	It is well known that if $G/Z$ is $p$-nilpotent and $Z\leq \Phi(P)$ then $G$ is $p$-nilpotent. Lemma \ref{main lemma}(b) generalizes this particular case by stating that if $\overline G/O^p(\overline G)\cong \overline  N/O^p(\overline N)$ then  $ G/O^p( G)\cong   N/O^p(N)$ where $\overline G=G/Z$ and $Z\leq \Phi(P)$.
	
	We also should note that in Lemma \ref{main lemma}, we prove little more than what we need here as we see that it may have other applications too.
\end{remark}

\begin{proposition}\label{main prop}

	Let $G$ be a group and $P\in Syl_p(G)$. Assume that for every characteristic subgroup of $P$ that contains  $Z_{p-1}(P)$ is weakly closed in $P$. Then $N_G(P)$ controls $p$-transfer.
\end{proposition}

\begin{proof}[Proof]
	We proceed by induction on the order $G$. Let $Z=Z_{p-1}(P)$. Then $N_G(Z)$ controls $p$-transfer in $G$ by  (\cite{4}, Theorem 14.4.2). If $N_G(Z)<G$ then $N_G(P)$ controls $p$-transfer with respect to group $N_G(Z)$ by induction applied to $N_G(Z)$. It follows that $P\cap G'=P\cap N_G(Z)'=P\cap N_G(P)'$, that is, $N_G(P)$ controls $p$-transfer in $G$.
	
	 Therefore we may assume $Z\lhd G$. It is easy to see that $G/Z$ satisfies the hypothesis of the proposition, and hence we get $N_{G/Z}(P/Z)=N_G(P)/Z$ controls $p$-transfer in $G/Z$ by induction applied to $G/Z$. Then the result follows by Lemma \ref{main lemma}(a).
\end{proof}

\begin{remark}
In the above proposition, the assumption that every characteristic subgroup containing $Z_{p-1}(P)$ is weakly closed can be weakened to $Z_{k(p-1)}(P)$ is weakly closed for each $k=1,...,n$ where $Z_{n(p-1)}(P)=P$. Yet we shall not need this fact.
\end{remark}
After Proposition \ref{main prop}, it is natural to ask the following question.
\begin{question}
Does a Sylow $p$-subgroup $P$ of a group $G$ have a  single characteristic subgroup whose being weakly closed in $P$ is sufficient to conclude that $N_G(P)$ controls $p$-transfer in $G$?
\end{question}

\begin{proof}[\textbf{Proof of Theorem \ref{main theorem}}]
	Let $Z_{p-1}\leq C$ be a characteristic subgroup of $P$. We claim that $C$ is normal in each Sylow subgroup of $G$ that contains $C$. Assume  the contrary and let $Q\in Syl_p(G)$ such that $C\leq Q$ and $N_Q(C)<Q$. There exists $x\in N_G(C)$ such that $N_Q(C)^x=N_{Q^x}(C)\leq P$, and hence $N_{Q^x}(C)\leq P\cap Q^x$.
	
	Set $Q^x=R$. By Alperin Fusion theorem, we have $R\sim_{P} P$. Thus there are Sylow subgroups $Q_i$ for $i=1,2,\ldots, n$ such that $P\cap R\leq P\cap Q_1$ and $(P\cap R)^{x_1x_2\ldots x_i}\leq P\cap Q_{i+1}$ where $x_i\in N_G(P\cap Q_{i})$, $P\cap Q_i$ is a tame intersection and $R^{x_1x_2...x_n}=P$.
	
	Note that $N_P(P\cap Q_1)$ is a Sylow $p$-subgroup of $N_G(P\cap Q_1)$ as $P\cap Q_1$ is a tame intersection. Moreover, $N_G(P\cap Q_1)$ is $p$-nilpotent by the hypothesis as $Z_{p-1}\leq C<N_Q(C)^x\leq  P\cap R\leq P\cap Q_1$. Then we have 
	$$N_G(P\cap Q_1)=N_P(P\cap Q_1)  C_G(P\cap Q_1).$$
	
	 Thus, we can write $x_1=s_1t_1$ where $t_1\in C_G(P\cap Q_1)$ and $s_1\in N_P(P\cap Q_1)$. Notice that $t_1$ also centralizes $C$ as $C\leq P\cap Q_1$ and $s_1$ normalizes $C$ as $C\unlhd P$. It follows that $C^{x_1}=C^{s_1t_1}=C<(P\cap R)^{x_1}\leq P\cap Q_2$. Then we get that $N_G(P\cap Q_2)$ is $p$-nilpotent by the hypothesis and we may write $x_2=s_2t_2$ where $t_2\in C_G(P\cap Q_2)$ and $s_2\in N_P(P\cap Q_2)$ in a similar way. Notice also that $C^{x_1x_2}=C^{x_2}=C$.  Proceeding inductively, we obtain that $N_G(P\cap Q_{i})$ is $p$-nilpotent for all $i$ and $C^{x_1x_2...x_n}=C$. Since $C^{x_1x_2...x_n}=C\lhd P=R^{x_1x_2...x_n}$, we get $C\lhd R=Q^x$. Since $x\in N_G(C)$, $C\lhd Q$. This contradiction shows that $C$ is weakly closed in $P$ and the theorem follows by Proposition \ref{main prop}.
\end{proof}

\begin{proof}[\textbf{Proof of Theorem \ref{nwthm}}]
Write $N=N_G(K)$, and let $X$ be a set of representatives for the $(N,P)$ double cosets in $G$, which contains the identity $e$. Note that $ N_G(P)\leq N$ as $K$ is a weakly closed subgroup of $P$. Assume that $N$ does not control $p$-transfer in $G$.
By Lemma \ref{basic}(b), we have a pretransfer map $W:P\to P\cap N^x$ such that $W(u)\notin P\cap M^x$ for each $u\in P\setminus M$ where $e\neq x\in X$ and $M$ is as in Lemma \ref{basic}(a). Set $R=P\cap N^x$ and $Q=P\cap M^x$.

Now choose $u\in P\setminus M$ and $u^*\in N\setminus M$ such that both $u$ and $u^*$ are of minimal possible order. We first argue that $|u|=|u^*|$. Clearly we have $|u^*|\leq |u|$ as $u\in N\setminus M$. Note that $(u^*)^q\in M$ if $q$ is a prime dividing the order $u^*$ by the choice of $u^*$. The previous argument shows that $p=q$ as $|N:M|=p$, and so $u^*$ is a $p$-element. Thus, a conjugate of $u^*$ via an element of $N$ lies in $P\setminus M$. It follows that $|u|\leq |u^*|$, which give us the desired equality. 

Let $S$ be a right transversal set for $R$ in $P$ used for constructing $W$ so that we have $W(u)=\prod_{s\in S}su(s.u)^{-1}.$ Let $S_0$ be a set of orbit representatives of the action of $\langle u \rangle$ on $S$. Then we have $W(u)=\prod_{s\in S_0}su^{n_s}s^{-1}$ by transfer evaluation lemma. Note that $su^{n_s}s^{-1}\in R\leq N^x$, and hence $xsu^{n_s}s^{-1}x^{-1}\in N$. If $n_s>1$ then $|xsu^{n_s}s^{-1}x^{-1}|<|u|$, and so $xsu^{n_s}s^{-1}x^{-1}\in M$ by the previous paragraph. Thus we get $su^{n_s}s^{-1}\in Q$. As a consequence, we observe that $$W(u)\equiv \prod_{s\in S^*} sus^{-1} \ mod \ Q$$
where $S^*=\{s\in S \mid s.u=s\}$.

We claim that $K$ is not contained in $R$. Since otherwise: both $K$ and $K^x$ are contained in $N^x$, and so $K^{x^{-1}}$ and $K$ are contained in $N$. Since $K$ is a weakly closed subgroup of $P$, there exists $y\in N$ such that $K^{x^{-1}}=K^y$ (see problem 5C.6(c) in \cite{1}). As a result $yx\in N$, and so $x\in N$. Thus, we get $NxP=NeP$ which is a contradiction as $x\neq e$. Since $R<P$ by Lemma \ref{basic}(c), $Core_P(R)$ is also proper in $P$. So we see that $Z(P/ Core_P(R))\neq 1$. Since $K$ is not contained in $Core_P(R)$ and $K$ is normal in $P$, we can pick $k\in K$ such that $kCore_P(R)\in Z(P/Core_P(R))$ of order $p$. Now consider the action of $\langle k \rangle$ on $S$. Then each $\langle k \rangle$-orbit has length $p$ and let $s_1,s_2\ldots,s_n$ be representatives of all distinct orbits of $\langle k \rangle$ on $S$. Note that we may replace $S$ with $\{s_ik^j \mid i\in\{1,2,...,n\} \ and \ j\in \{0,1,...,p-1\}\}$. We also note that $$s.(uk)=(s.(ku)).[u,k]=s.(ku) \ \textit{for all }s\in S.$$

The last equality holds as $[u,k]\in Core_P(R)$. It follows that $S^*$ is $\langle k \rangle$-invariant. Note that $k$ normalizes $\langle u \rangle$ as $k\in Z^*(P)$, and so $u^{k^{-1}}=u^n$ where $n$ is a natural number which is coprime to $p$. Clearly  $n$ is odd when $p=2$. On the other hand, if $p$ is odd then it is well known that $n=(1+p)^r$ for some $r\in \mathbb N$ as $k^{-1}$ induces a $p$-automorphism on a cyclic $p$-group. Thus, we obtain $n\equiv 1 \ mod \ p$ in both case.

Now we compute the contribution of a single $\langle k \rangle$-orbit to $W(u)$. Fix $s\in S^*$.

$(sus^{-1})(skuk^{-1}s^{-1})(sk^2uk^{-2}s^{-1})...(sk^{p-1}uk^{-p+1}s^{-1})=suu^nu^{n^2}...u^{n^{p-1}}s^{-1}=su^zs^{-1}$ where $z=1+n+n^2+...+n^{p-1}$. Note that $z \equiv 0 \ mod \ p$, $sus^{-1}\in R$ and $|R:Q|=p$ by Lemma \ref{basic}(c), and hence $su^zs^{-1}=(sus^{-1})^z \in Q$. Since the chosen $\langle k \rangle$-orbit is arbitrary, we obtain $W(u)\in Q$. This contradiction completes the proof.
\end{proof}
Now we are ready to give the proof of Theorem \ref{Z version}.
\begin{proof}[\textbf{Proof of Theorem \ref{Z version}}]
	 First notice that if $p$ is odd then the result follows by Theorem \ref{main theorem} due to the fact that $Z^*(P)\leq Z_2(P)\leq Z_{p-1}(P)$. Thus, it is sufficient to prove the theorem for $p=2$. 
Let $G$ be a minimal counter example to the theorem. We derive a contradiction over a series of steps. Write $Z=Z^*(P)$ and $N=N_G(P)$.\vspace{0.2 cm}

$\textbf{(1)}$ Each characteristic subgroup $C$ of $P$ that contains $Z$ is weakly closed in $P$.  Moreover, $Z$ is a normal subgroup of $G$.\vspace{0.2 cm}

By using the same strategy used in the proof of Theorem \ref{main theorem}, we can show that any characteristic subgroup $C$ of $P$ that contains $Z$ is weakly closed in $P$. In particular, $Z$ is weakly closed in $P$.

Suppose that $N_G(Z)<G$. Clearly $N_G(Z)$ satisfies the hypothesis and $N\leq N_G(Z)$. Thus, $N$ controls $p$-transfer with respect to the group $N_G(Z)$ by the minimality of $G$. On the other hand, $N_G(Z)$ controls $p$-transfer in $G$ by Theorem \ref{nwthm}. As a consequence, $G'\cap P=(N_G(Z))'\cap P=N'\cap P$. This contradiction shows that $Z\lhd G$.\vspace{0.2 cm}

$\textbf{(2)}$ $N/Z$ controls $p$-transfer in $G/Z$.
\vspace{0.2 cm}

 Write $\overline G=G/Z$. Clearly $\overline N=N_{\overline G}(\overline P)$. If $\overline Y$ is a characteristic subgroup of $\overline P$ then $Y$ is a characteristic subgroup of $P$ that contains $Z$. Then $Y$ is weakly closed in $P$ by (1). It follows that $\overline{Y}$ is weakly closed in $\overline P$.  Then we get $\overline N$ controls $p$-transfer in $\overline G$ by Proposition \ref{main prop}.

\vspace{0.2 cm}
$\textbf{(3)}$  $|P:R|=2$.\vspace{0.2 cm}

By Lemma \ref{u,Z}, there exists  $u\in Z\setminus M $ such that $W(u) \in P \cap N^x \setminus P\cap M^x$ where $W,M$ and $x$ are as in Lemma \ref{basic}. Set $R=P \cap N^x$ and $Q= P\cap M^x$. Let $S$ be a right transversal set for $R$ in $P$ used for constructing $W$. Since $u\in Z\leq Core_P(R)$, we get $W(u)=\prod_{s\in S}su(s.u)^{-1}=\prod_{s\in S}sus^{-1}.$

Since $R<P$ by Lemma \ref{basic}(c), $Core_P(R)$ is also proper in $P$. So we see that $Z(P/ Core_P(R))\neq 1$. Now choose $n\in P$ such that $n Core_P(R)\in  Z(P/ Core_P(R))$ of order $p$ and consider the action of $\langle n \rangle$ on $S$. Without loss of generality, we may take $S=\{s_in^j \mid i\in\{1,2,...,k\} \ and \ j\in \{0,1\}\}$ where $s_1,s_2\ldots,s_k$ are representatives of all distinct orbits of $\langle n \rangle$ on $S$. Fix $s\in S$. We have
$$(sus^{-1})(snun^{-1}s^{-1})=su^2[u,n^{-1}]s^{-1}=su^2s^{-1}[u,n^{-1}].$$
The last equality holds as $u\in Z=Z^*(P)\leq Z_2(P)$. We see that $su^2s^{-1}\in Q$ as $sus^{-1}\in Z\leq R$ and $|R:Q|=2$. Thus the contribution of a single orbit is congruent to $[u,n^{-1}]$ at mod $Q$ by Lemma \ref{basic}(c). As a consequence, we obtain that $W(u)\equiv [u,n^{-1}]^{|S|/2} \ mod \ Q$. Suppose that $|S|/2$ is an even number. We get $[u,n^{-1}]^{|S|/2}\in Q$ as $[u,n^{-1}]\in Z\leq R$. This contradicts the fact that $W(u)\notin Q$, and so $|S|/2$ is odd. It follows that $|P:R|=|S|=2$ as required.\vspace{0.2 cm}

$\textbf{(4)}$ $R=Z$.\vspace{0.2 cm}

Suppose that $Z<R$. Note that $R=P\cap N^x=P\cap N_G(P)^x$, and so $R=P\cap P^x$. Since $|P:R|=2$ by (3), $|P^x:R|$ is also equal to $2$. As a result, $R$ is normal in both $P$ and $P^x$, that is, $R$ is a tame intersection. Thus, we see that $N_G(R)$ is $p$-nilpotent by hypothesis. Pick $x_0\in N_G(R)$ such that $P^x=P^{x_0}$. Then $x_0x^{-1}\in N$ which implies $x_0=tx$ for some $t\in N$. We observe that $Nx_0P=NtxP=NxP$, and so we may replace the double coset representative $x$ with $x_0$.

Since $N_G(R)$ is $p$-nilpotent, we can write $x=c_1c_2$ for some $c_1\in P$ and $c_2\in C_G(R)$. As $W(u)\notin Q=P\cap M^x$, we see that $xW(u)x^{-1}=c_1c_2W(u)c_2^{-1}c_1^{-1}=c_1W(u)c_1^{-1}\notin M$. Thus, $W(u)\notin M^{c_1}=M$. Recall that $|P:M \cap P|=p=2$, and so $P'\leq M$. Hence, we obtain that
$$W(u)=\prod_{s\in S}sus^{-1}=\prod_{s\in S} [s^{-1},u^{-1}]u\equiv \prod_{s\in S}u =u^2\equiv 1 \ mod \ M\cap P .$$ It follows $W(u)\in M$, which is not the case. This contradiction shows that $Z=R$.\vspace{0.2 cm}

$\textbf{(5)}$ Final contradiction.\vspace{0.2 cm}

We observe that $|P:Z|=|P:Z^*(P)|=2$ by $(4)$. If $\overline P$ is a homomorphic image of $P$, we can conclude that $|\overline P:Z^*(\overline P)|\leq 2$. Since $N$ does not control $p$-transfer in $G$, $P$ has a homomorphic image which is isomorphic to $\mathbb Z_2 \wr \mathbb Z_2\cong D_8$ by Yoshida's theorem. However, $|D_8:Z^*(D_8)|=|D_8:Z(D_8)|=4$. This contradiction completes the proof.
\end{proof}

\section{Applications}

	\begin{theorem}\label{p,p-1}
		Assume that for any two distinct Sylow $p$-subgroups $P$ and $Q$ of $G$, $|P\cap Q|\leq p^{p-1}$. Then $N_G(P)$ controls $p$-transfer in $G$.
	\end{theorem}

\begin{proof}[\textbf{Proof}]
We may suppose that $cl(P)\geq p$.	Notice that the inequality $|Z_{p-1}(P)|\geq p^{p-1}$ holds in this case. Then the result follows by Corollary \ref{main cor}.
\end{proof}

The main theorem of \cite{3} states that if $N_G(P)$ is $p$-nilpotent and for any two distinct Sylow $p$-subgroups $P$ and $Q$ of $G$, $|P\cap Q| \leq p^{p-1}$ then $G$ is $p$-nilpotent. The above theorem is a generalization of this fact.
	\begin{theorem}\label{app class p}
		Let $P\in Syl_p(G)$. Suppose that $P$ is of classes $p$ and $N_G(P)$ is $p$-nilpotent. If $N_G(P)$ is a maximal subgroup of $G$ then $G$ is a $p$-solvable group of length $1$.
	\end{theorem}
	
	\begin{proof}[\textbf{Proof}]
		We may suppose that $G$ is not $p$-nilpotent. Then there exists $U\leq G$ such that $Z_{p-1}<U<P$ and $N_G(U)$ is not $p$-nilpotent by Corollary \ref{gen Frob}. Since $Z_{p-1}<U$, $U\unlhd P$. It follows that $U\unlhd N_G(P)$ as $N_G(P)$ is $p$-nilpotent. Note that $N_G(P)\neq N_G(U)$ as $N_G(U)$ is not $p$-nilpotent. Thus we get $N_G(P)<N_G(U)$, and hence $U\lhd G$. On the other hand, $G/U$ is $p$-nilpotent as $P/U$ is an abelian Sylow subgroup of $G/U$ where $N_G(P)/U=N_{G/U}(P/U)$ is $p$-nilpotent. Then the result follows.
	\end{proof}
	
	\begin{theorem}
Let $P\in Syl_p(G)$. Suppose that $P$ is of class $p$ and the number of Sylow $p$-subgroups of $G$ is $p+1$. Then either $N_G(P)$ controls $p$-transfer in $G$ or $O_p(G)\neq 1$.
	\end{theorem}

\begin{proof}[\textbf{Proof}]
	Suppose that $N_G(P)$ does not control $p$-transfer in $G$. Then there exists a tame intersection $Z_{p-1}<P\cap Q<P$ by Theorem \ref{main theorem}. Since $P\cap Q \lhd P$ and $P\cap Q$ is a tame intersection, we have also $P\cap Q \lhd Q$. It then follows that $P\cap Q \lhd \langle P,Q\rangle$. Due to the fact that $G$ has $p+1$ Sylow $p$-subgroups and $P\neq Q$, $\langle P,Q\rangle= \langle P^g\mid g\in G \rangle$, and hence it is a normal subgroup of $G$. Then we obtain  that $1<P\cap Q\leq  O_p(\langle P,Q\rangle )\leq O_p(G)$ as desired.
\end{proof}
	
	Thompson proved that if $G$ posses a nilpotent maximal subgroup of odd order then $G$ is solvable. Later Janko extended this result in \cite{5} as follows;
	
	\begin{theorem}[Janko]
		Let $G$ be a group having a nilpotent maximal subgroup $M$. If a Sylow $2$-subgroup of $M$ is of class at most $2$ then $G$ is solvable.
	\end{theorem}
The above theorem can be deduced by the means of Theorem \ref{app class p}. We extend the result of Janko by using Corollary \ref{gen frob Z-version} with the following theorem.
	\begin{theorem}
		Let $G$ be a group with a nilpotent maximal subgroup $M$. If a Sylow $2$-subgroup of $M$ is of norm length at most $2$ then $G$ is solvable.
	\end{theorem}
	\begin{proof}[\textbf{Proof}]
		We proceed by induction on the order of $G$.  Suppose $O_p(G) \neq 1$ for a prime $p$ dividing the order of $M$. If $O_p(G)\leq M$ then $G/O_p(G)$ satisfies the hypothesis and hence $G/O_p(G)$ is solvable by induction. If $O_p(G)\nleq M$ then $G=MO_p(G)$ due to the maximality of $M$. Thus, $G/O_p(G)$ is solvable as $M$ is nilpotent. Then we see that $G$ is solvable in both cases. Thus, we may suppose that $O_p(G)= 1$ for any prime $p$ dividing the order of $M$.
		
		Now let $P\in Syl_p(M)$. Since $M$ is nilpotent, we get $M\leq N_G(P)$. On the other hand, $N_G(P)<G$ as $O_p(G)=1$. Then we have $N_G(P)=M$ by the maximality of $M$. Thus $P$ is also a Sylow $p$-subgroup of $G$, that is, $M$ is a Hall subgroup of $G$. Let $X$ be a characteristic subgroup of $P$. Then $N_G(X)=M$ with a similar argument, and hence $N_G(X)$ is $p$-nilpotent. It follows that $G$ is $p$-nilpotent by Thompson $p$-nilpotency theorem when $p$ is odd. 
		
		Now assume that $p=2$. Let $Z^*(P)\leq U\leq P=Z^*_2(P)$. Since $P/Z^*(P)$ is a Dedekind group, $U/Z^*(P)\unlhd P/Z^*(P)$. It follows that $U\unlhd P$, and hence $U\lhd M$. Then we get $N_G(U)=M$ which is $p$-nilpotent. Thus, we obtain that $G$ is $p$-nilpotent by Corollary \ref{gen frob Z-version}.
		
		As a result $G$ is $p$-nilpotent for each prime $p$ dividing the order of $M$. Then $M$ has a normal complement $N$ in $G$. Notice that $M$ acts on $N$ coprimely, and so we may choose an $M$-invariant Sylow $q$-subgroup $Q$ of $N$ for a prime $q$ dividing the order of $N$. The maximality of $M$ forces that $MQ=G$, that is, $N=Q$. Since $N$ is a $q$-group, we see that $G$ is solvable.
	\end{proof}
	\begin{remark}
		We should note that there are groups of class $3$, which have norm length $2$. For example, one can consider the quaternion group $Q_{16}$. We also note that the bound in terms of norm length is the best possible. For example, $D_{16}$ is of norm length $3$ and it is isomorphic to a Sylow $2$-subgroup $P$ of $PSL(2,17)$ and $P$ is a maximal subgroup of $G$.
	\end{remark}

	\begin{definition}
		A group $G$ is called $p^i$-central of height $k$ if every element of order $p^i$ of $G$ is contained in $Z_k(G)$.
	\end{definition}

	\begin{theorem}\label{generalized p central}
		Let $G$ be a group and $P$ be a Sylow $p$-subgroup of $G$ where $p$ is an odd prime. Assume that either $P$ is $p$-central of height $p-2$ or $p^2$-central of height of $p-1$. Then $N_G(P)$ controls $p$-transfer in $G$.
	\end{theorem}
	
	\begin{remark}
		Let $G$ be a group and $P\in Syl_p(G)$. Assume that $P$ is $p$-central of height $p-2$ for an odd prime $p$. 
		By (\cite{6}, Theorem E), $N_G(P)$ controls $G$-fusion if $G$ is a $p$-solvable group. In this case, $N_G(P)$ also controls $p$-transfer in $G$. On the other hand, Theorem \ref{generalized p central} guarantees that $N_G(P)$ controls $p$-transfer in $G$ for an arbitrary finite group $G$.
	\end{remark}
We need the following result in the proof of Theorem \ref{generalized p central}.

\begin{theorem}\cite[Theorem B]{6}\label{theoremB}
	Let $G$ be a group. If $G$ is $p$-central of height $p-2$ or $p^2$-central of height of $p-1$, then so is $G/\Omega(G)$.
\end{theorem}

\begin{proof}[\textbf{Proof of Theorem \ref{generalized p central}}]
	We proceed by induction on the order $G$. Set $Z=\Omega(P)$.		
	Clearly, $Z$ is weakly closed in $P$. Since $\Omega(P)\leq Z_{p-1}(P)$, $N_G(Z)$ controls $p$-transfer in $G$ by (\cite{4}, Theorem 14.4.2).
	
	 If $N_G(Z)<G$ then $N_G(Z)$ clearly satisfies the hypothesis, and hence $N_G(P)$ controls $p$-transfer in $N_G(Z)$. It follows that $P\cap G'=P\cap N_G(Z)'=P\cap N_G(P)'$, and hence $N_G(P)$ controls $p$-transfer in $G$.
	 
	  Now assume that $Z\unlhd G$. By Theorem \ref{theoremB}, $P/Z$ is a Sylow $p$-subgroup of $G/Z$, which is $p$-central of height $p-2$ or $p^2$-central of height of $p-1$. Thus, $N_G(P)/Z=N_{G/Z}(P/Z)$ controls $p$-transfer in $G/Z$ by induction. Since $Z\leq Z_{p-1}(P)$, the result follows by Lemma \ref{main lemma}.
\end{proof}

\begin{conclusion*}
	
	``Control $p$-transfer theorems" supply many nonsimplicity theorems by their nature.
	Let $N$ be a subgroup of a group $G$ such that $|G:N|$ is coprime to $p$. If $N$ controls $p$-transfer in $G$ and $O^p(N)<N$ then $G$ is not simple of course.
	
	It is an easy exercise to observe that if $K$ is a normal $p'$-subgroup of $G$, and write $\overline G=G/K$, then $\overline N$ controls $p$-transfer in $\overline G$ if and only if $N$ controls $p$-transfer in $G$. However, this need not be true if $K$ is a $p$-group. Thus, Lemma \ref{main lemma} supplies an important criterion for that purpose and it enables the usage of the induction in  the proof of ``Control $p$-transfer theorems". It also seems that Lemma \ref{Z version} can be improved further by better commutator tricks or more careful analysis of the transfer map. 
	
	Proposition \ref{main prop} shows that when some certain characteristic subgroups of a Sylow subgroup are weakly closed in $P$, $N_G(P)$ controls $p$-transfer in $G$. One can ask that whether the converse of this statement is true? Another natural question is that whether ``control fusion" analogue of Lemma \ref{main lemma} and Proposition \ref{main prop} are possible.
	
	When we combine Proposition \ref{main prop} with Alperin Fusion theorem, we obtain our main theorems, which simply say that $N_G(P)$ tends to controls $p$-transfer in $G$ if intersection of Sylow subgroups is not ``too big". We also sharpen our result when $p=2$ via Theorem \ref{Z version} and deduce two new versions of Frobenius normal complement theorem namely, Corollary \ref{gen Frob} and Corollary \ref{gen frob Z-version}. Since, we can not directly appeal to Thompson-Glauberman $p$-nilpotency theorems when $p=2$ (and $G$ is not $S_4$ free), the contribution of Corollary \ref{gen frob Z-version} is important.
	
	Besides the other applications, Theorem \ref{generalized p central} shows that $N_G(P)$ controls $p$-transfer for groups which have Sylow subgroup isomorphic to one of the two important classes of $p$-groups, namely, $p$-central of height $p-2$ or $p^2$-central of height of $p-1$.
	 
	Even if we supply some limited applications here, we think that above theorems have nice potential of proving nonsimplicity theorems in finite group theory.

\end{conclusion*}

\section*{Acknowledgements}
I would like to thank Prof. George Glauberman for his helpful comments.

\end{document}